\documentclass[10pt]{article}
\font\smallit=cmti10

\usepackage{amssymb,latexsym,amsmath,epsfig,amsthm}

\makeatletter

\renewcommand\section{\@startsection {section}{1}{\z@}
{-30pt \@plus -1ex \@minus -.2ex}
{2.3ex \@plus.2ex}
{\normalfont\normalsize\bfseries\boldmath}}

\renewcommand\subsection{\@startsection{subsection}{2}{\z@}
{-3.25ex\@plus -1ex \@minus -.2ex}
{1.5ex \@plus .2ex}
{\normalfont\normalsize\bfseries\boldmath}}

\renewcommand{\@seccntformat}[1]{\csname the#1\endcsname. }

\makeatother

\newtheorem{theorem}{Theorem}

\newtheorem{corollary}{Corollary}

\theoremstyle{definition}

\begin{document}

\begin{center}
\uppercase{\bf \boldmath On Quotients of a More General Theorem of Wilson}
\vskip 20pt
{\bf Ivan V. Morozov}\\
{\smallit The City College of New York, New York, New York}\\
{\tt imorozo000@citymail.cuny.edu}
\end{center}
\vskip 30pt

\centerline{\bf Abstract}
\noindent
The basis of this work is a simple, extended corollary of Wilson’s theorem. This corollary generates many more quotients than those already generated by Wilson’s theorem, and it was of interest to derive how they relate to each other and build on the established properties of the original quotients. The most important results that were found were expressions for sums of these quotients, modular congruences that extended the results of Lehmer, and generating functions.

\section{Introduction}

In order to efficiently describe results presented in this work, we will define $\mathbb{P}_{1}:=\mathbb{P}\cup\{1\}$, where $\mathbb{P}$ is the set of all primes, as the set of positive non--composite numbers. With this said, Wilson's theorem is a primality test by which
\begin{equation}
W(n)=\frac{1+(n-1)!}{n}
\end{equation}
is an integer if and only if $n\geq 1$ is a positive non--composite. Subsequently, $\{W(p)\mid p\in\mathbb{P}_{1}\}$ is the set of Wilson quotients. We may, in fact, generalise this notion by observing that
\begin{equation}
(n-1)!=(n-k-1)!\prod_{m=1}^{k}(n-m)\Rightarrow W(n)=\frac{1+(n-k-1)!\prod_{m=1}^{k}(n-m)}{n}\,.
\end{equation}
The $\prod_{m=1}^{k}(n-m)$ term is a falling factorial, so we can implement an expansion of this product, the coefficients of which are the Stirling numbers of the first kind $s(a,b)$, defined as $(x)_{n}=\prod_{k=0}^{n-1}(x-k)=\sum_{k=0}^{n}s(n,k)x^{k}$. Thus $\prod_{m=1}^{k}(n-m)=\frac{1}{n}(n)_{k+1}$ and
\begin{equation}
W(n)=\frac{1+(n-k-1)!\frac{1}{n}(n)_{k+1}}{n}=\frac{1+(n-k-1)!\sum_{i=0}^{k+1}s(k+1,i)n^{i-1}}{n}\,,
\end{equation}
and since $s(k+1,0)=0$ for $k\geq 0$, then
\begin{equation}\label{wilson}
W(n)=\frac{1+(n-k-1)!\sum_{i=0}^{k+1}s(k+1,i+1)n^{i}}{n}\,.
\end{equation}
Observing that $\sum_{i=1}^{k+1}s(k+1,i+1)n^{i}\equiv 0\pmod{n}$ and $s(k+1,1)=(-1)^{k}k!$, Wilson's criterion extends to the following corollary.

\begin{corollary}\label{1}
Given non--negative integers $n\geq 1$ and $k<n$,
\[
(-1)^{k}k!(n-k-1)!\equiv -1\pmod{n}
\]
if and only if $n$ is non--composite.
\end{corollary}

Incidentally, Wilson's theorem is a corollary of its own corollary when $k=0$, and we will thus refer to the more general Wilson--like quotients as
\begin{equation}
M_{k}(n)=\frac{1+(-1)^{k}k!(n-k-1)!}{n}\,,
\end{equation}
where by Corollary \ref{1} $M_{k}(n)$ is an integer if and only if $n\in\mathbb{P}_{1}$. Note that $M_{0}(n)=W(n)$. These quotients yield non--positive results for odd $k$, and therefore we can get rid of the signs by defining
\begin{equation}\label{14}
M_{k}^{+}(n)=\lvert M_{k}(n)\rvert=(-1)^{k}M_{k}(n)=\frac{(-1)^{k}+k!(n-k-1)!}{n}\,.
\end{equation}

\section{Sums of $M_{k}(n)$ and $M_{k}^{+}(n)$}

One area of properties to investigate is the nature of sums of these quotients over varying $k$. Considering the sums
\begin{equation}\label{zski}
Z(n)=\sum_{k=0}^{n-1}M_{k}(n)\,,\quad Z^{+}(n)=\sum_{k=0}^{n-1}M_{k}^{+}(n)\,,
\end{equation}
we can derive the following two theorems concerning them.

\begin{theorem}\label{thm}
$Z(n)$ is an integer for all $n\in\mathbb{N}$.
\end{theorem}

\begin{proof}
If $n\in\mathbb{P}_{1}$, then $M_{k}(n)\in\mathbb{Z}\Rightarrow Z(n)\in\mathbb{Z}$ by Corollary \ref{1}. Otherwise, consider the term $(-1)^{k}k!(n-k-1)!$, where $n$ is composite. If it is additionally not equal to $4$, we see that
\[
k!(n-k-1)!=k!(n-(k+1))(n-(k+2))\ldots (n-(n-1))
\]
\[
k!(n-k-1)!\equiv (-1)^{k}k!(k+1)(k+2)\ldots (n-1)\equiv (-1)^{k}(n-1)!\pmod{n}.
\]
Moreover, since $n\neq 4$, $(n-1)!\equiv 0\pmod{n}$, implying $(-1)^{k}k!(n-k-1)!\equiv 0\pmod{n}$. This implies that
\[
\sum_{k=0}^{n-1}1+(-1)^{k}k!(n-k-1)!\equiv \sum_{k=0}^{n-1}1\equiv 0\pmod{n}\,,
\]
which, with the fact that $Z(4)=1$, shows that $Z(n)\in\mathbb{Z}$ for all $n\in\mathbb{N}$.
\end{proof}

\begin{theorem}
$Z^{+}(n)$ is a natural number if and only if $n$ is even or non--composite.
\end{theorem}

\begin{proof}
If $n\in\mathbb{P}_{1}$, then $M^{+}_{k}(n)\in\mathbb{Z}\Rightarrow Z^{+}(n)\in\mathbb{Z}$ by Corollary \ref{1}. Otherwise, consider the term $(-1)^{k}k!(n-k-1)!$, where $n$ is composite. If it is additionally not equal to $4$, the preceeding proof showed that
\[
k!(n-k-1)!\equiv 0\pmod{n}.
\]
This implies that
\[
\sum_{k=0}^{n-1}(-1)^{k}+k!(n-k-1)!\equiv \sum_{k=0}^{n-1}(-1)^{k}\equiv\frac{1+(-1)^{n-1}}{2}\pmod{n}\,,
\]
which, with the fact that $Z^{+}(4)=4$, shows that $Z^{+}(n)\in\mathbb{Z}$, and moreover $Z^{+}(n)\in\mathbb{N}$ since $Z^{+}(n)>0$, if and only if $n$ is even or non--composite.
\end{proof}

\subsection{Formula for $Z(n)$}

Furthermore, we can derive a closed formula for $Z(n)$. Let $S=\sum_{k=0}^{n-1}(-1)^{k}k!(n-k-1)!$. Then $Z(n)=\left(\sum_{k=0}^{n-1}\frac{1}{n}\right)+\frac{S}{n}=1+\frac{S}{n}$. Now, dividing $S$ by $(n-1)!$ yields $\frac{S}{(n-1)!}=\sum_{k=0}^{n-1}\frac{(-1)^{k}k!(n-k-1)!}{(n-1)!}=\sum_{k=0}^{n-1}\frac{(-1)^{k}}{\binom{n-1}{k}}$. An identity for the same sum \cite{sury} tells us that
\begin{equation}
\sum_{k=0}^{n}\frac{(-1)^{k}}{\binom{n}{k}}=\frac{\left(1+(-1)^{n}\right)(n+1)}{n+2}\,,\quad n\in\mathbb{N}_{0}\,,
\end{equation}
so $\frac{S}{(n-1)!}=\frac{\left(1+(-1)^{n-1}\right)n}{n+1}$, and it follows algebraically that
\begin{equation}
Z(n)=1+\frac{\left(1+(-1)^{n-1}\right)(n-1)!}{n+1}\,.
\end{equation}
We observe that $Z(2n)=1$ and $Z(2n-1)=1+\frac{(2n-2)!}{n}$. The former case is obviously a natural number. In the latter case, $n=1$ is a trivial case, and since $2n-2\geq n$ for $n\geq 2$, $n\mid (2n-2)!$ for all $n\in\mathbb{N}$. This is a proof of a stronger version of Theorem \ref{thm}: $Z(n)\in\mathbb{N}$ for all $n\in\mathbb{N}$.

\subsection{Formula for $Z^{+}(n)$}

We can also derive a formula for $Z^{+}(n)$ by initially setting $S=\sum_{k=0}^{n-1}k!(n-k-1)!$, which produces $Z^{+}(n)=\left(\sum_{k=0}^{n-1}\frac{(-1)^{k}}{n}\right)+\frac{S}{n}=\frac{1+(-1)^{n-1}}{2n}+\frac{S}{n}$. Dividing $S$ by $(n-1)!$ produces $\frac{S}{(n-1)!}=\sum_{k=0}^{n-1}\frac{k!(n-k-1)!}{(n-1)!}=\sum_{k=0}^{n-1}\frac{1}{\binom{n-1}{k}}$. An identity from \cite{sury} tells us that
\begin{equation}
\sum_{k=0}^{n}\frac{1}{\binom{n}{k}}=\frac{n+1}{2^{n}}\sum_{r=0}^{n}\frac{2^{r}}{r+1}\,,\quad n\in\mathbb{N}_{0}\,,
\end{equation}
so $\frac{S}{(n-1)!}=\frac{n}{2^{n-1}}\sum_{r=0}^{n-1}\frac{2^{r}}{r+1}$, and it follows algebraically that
\begin{equation}
Z^{+}(n)=\frac{1+(-1)^{n-1}}{2n}+\frac{(n-1)!}{2^{n}}\sum_{r=1}^{n}\frac{2^{r}}{r}\,.
\end{equation}
To deal with the remaining partial sum, we can make use of the Lerch transcendent per the equivalence $\Phi(z,s,a)=z^{n}\Phi(z,s,n+a)+\sum_{r=0}^{n-1}\frac{z^{r}}{(r+a)^{s}}$ \cite{olbc} for $\Re(a),\Re(s)>0$, $n\in\mathbb{N}$, and $z\in\mathbb{C}$. Setting $a,s=1$ and $z=2$,
\begin{equation}
\Phi(2,1,1)=2^{n}\Phi(2,1,n+1)+\sum_{r=0}^{n-1}\frac{2^{r}}{r+1}\,.
\end{equation}
Because $\Phi(z,s,1)=\frac{1}{z}\mathrm{Li}_{s}(z)$ \cite{guison},
\begin{equation}
\mathrm{Li}_{1}(2)=-i\pi=2^{n+1}\Phi(2,1,n+1)+\sum_{r=1}^{n}\frac{2^{r}}{r}
\end{equation}
\begin{equation}
\sum_{r=1}^{n}\frac{2^{r}}{r}=-i\pi-2^{n+1}\Phi(2,1,n+1)\,.
\end{equation}
By substitution we derive that
\begin{equation}\label{zpluslerch}
Z^{+}(n)=\frac{1+(-1)^{n-1}}{2n}-(n-1)!\left(2\Phi(2,1,n+1)+2^{-n}i\pi\right)\,.
\end{equation}

\section{Modular Congruences}

\subsection{Congruences of $M_{k}(p)$ and $M^{+}_{k}(p)$}

Bernoulli numbers $B_{k}$ are signed rational numbers that are ubiquitous in number theory and analysis. We can define them precisely using the exponential generating function $\frac{x}{e^{x}-1}=\sum_{k=0}^{\infty}\frac{B_{k}x^{k}}{k!}$, in which they arise. When $p$ is prime, it is known that
\begin{equation}
M_{0}(p)\equiv W(p)\equiv B_{2(p-1)}-B_{p-1}\pmod{p}\,,
\end{equation}
which is obtained from the congruence relation
\begin{equation}\label{lehmer}
p-1+ptW(p)\equiv pB_{t(p-1)}\pmod{p^{2}}
\end{equation}
by subtraction after substituting $t=1$ and $t=2$ \cite{lehmer}. Recalling \eqref{wilson}, it is clear that $\frac{(p-k-1)!\sum_{i=2}^{k+1}s(k+1,i+1)p^{i}}{p}\equiv 0\pmod{p}$, so
\begin{multline}
M_{0}(p)\equiv\frac{1+(p-k-1)!(s(k+1,1)+ps(k+1,2))}{p}\\\equiv M_{k}(p)+s(k+1,2)(p-k-1)!\pmod{p}\,.
\end{multline}
The Stirling numbers of the first kind $s(k+1,2)$ can be expressed with harmonic numbers as $s(k+1,2)=(-1)^{k+1}k!H_{k}$, which yields
\begin{equation}\label{sub}
M_{0}(p)\equiv M_{k}(p)+(-1)^{k+1}k!H_{k}(p-k-1)!\equiv M_{k}(p)+H_{k}\pmod{p}
\end{equation}
by Corollary \ref{1}. By subtraction,
\begin{equation}\label{mcong}
M_{k}(p)\equiv B_{2(p-1)}-B_{p-1}-H_{k}\pmod{p}\,.
\end{equation}
From \eqref{sub}, we can construct a congruence analogous to Lehmer's,
\begin{equation}
p-1+ptM_{k}(p)\equiv pB_{t(p-1)}-ptH_{k}\pmod{p^{2}}\,.
\end{equation}
For the unsigned $M$--numbers, utilising \eqref{14} produces the congruences
\begin{equation}\label{mpluscong}
M^{+}_{k}(p)\equiv (-1)^{k}M_{k}(p)\equiv (-1)^{k}\left(B_{2(p-1)}-B_{p-1}-H_{k}\right)\pmod{p}\,,
\end{equation}
\begin{equation}
(-1)^{k}(p-1)+ptM^{+}_{k}(p)\equiv (-1)^{k}\left(pB_{t(p-1)}-ptH_{k}\right)\pmod{p^{2}}\,.
\end{equation}

\subsection{Congruences of $Z(n)$ and $Z^{+}(n)$}

Recalling the formulae $Z(2n)=1$ and $Z(2n-1)=1+\frac{(2n-2)!}{n}$ for $n\in\mathbb{N}$, we obtain an interesting congruence relation for $Z(n)$ by noticing that if $n=2k-1$ is composite, then we can apply the same argument as in the proof of Theorem \ref{thm} to show that $2k-1\mid (2k-2)!$ and thus $2k-1\mid\frac{(2k-2)!}{k}$ since $k\nmid 2k-1$, implying $Z(n)\equiv 1\pmod{n}$ for $n\not\in\mathbb{P}_{1}$. However, if $n=2k-1$ is non--composite, then $k(Z(2k-1)+1)=2k+(2k-2)!\equiv 1-1\equiv 0\pmod{2k-1}$ by Wilson's theorem, and thus, since $k\nmid 2k-1$, $Z(2k-1)+1\equiv 0\pmod{2k-1}\Rightarrow Z(n)\equiv -1\pmod{n}$ for $n\in\mathbb{P}_{1}$. Overall, this results in the congruence
\begin{equation}
Z(n)\equiv(-1)^{\delta_{n\,\mathbb{P}_{1}}}\pmod{n}\,,
\end{equation}
where $\delta_{n\,\mathbb{P}_{1}}=1$ if $n\in\mathbb{P}_{1}$ and $\delta_{n\,\mathbb{P}_{1}}=0$ if $n\not\in\mathbb{P}_{1}$. By \eqref{mcong} we have $\sum_{k=1}^{p-1}H_{k}\equiv\sum_{k=1}^{p-1}\left(M_{0}(p)-M_{k}(p)\right)\equiv(p-1)M_{0}(p)-Z(p)+M_{0}(p)\equiv pM_{0}(p)-Z(p)\equiv 2+(p-1)!\pmod{p}$. Since $1+(p-1)!\equiv 0\pmod{p}$ by Wilson's theorem,
\begin{equation}
\sum_{k=1}^{p-1}H_{k}\equiv 1\pmod{p}
\end{equation}
for $p$ prime. We can also consider the recursive nature of harmonic numbers given by $H_{n}=1+\frac{1}{n}\sum_{k=1}^{n-1}H_{k}$, which by recursion breaks down into the true statement $H_{n+1}=\frac{1}{n+1}+H_{n}$. Substituting this into the previous sum implies the following congruence for $p$ prime,
\begin{equation}
H_{p}\equiv 1+\frac{1}{p}\pmod{p}\,.
\end{equation}

Moreover, we can derive another cute congruence for $Z(2n-1)$ modulo $n$. Since $Z(2n-1)=1+\frac{(2n-2)!}{n}=1+(2n-2)(2n-3)\cdots (n-1)!$, if $n\in\mathbb{P}_{1}$, then $Z(2n-1)\equiv 1-\frac{(2n-2)!}{n!}\pmod{n}$ by Wilson's theorem. $\frac{(2n-2)!}{n!}\equiv (n-2)!\pmod{n}$, so $Z(2n-1)\equiv 1-(n-2)!\pmod{n}$, and thus
\begin{equation}
Z(2n-1)\equiv 0\pmod{n}\,,\quad n\in\mathbb{P}_{1}
\end{equation}
by Corollary \ref{1}.

For $Z^{+}(n)$, we have two peculiar congruence relations, namely
\begin{equation}\label{c1}
Z^{+}(2n)\equiv 0\pmod{2n}\,,\quad n\not\in\mathbb{P}_{1}\setminus\{2\}\,,
\end{equation}
\begin{equation}\label{c2}
Z^{+}(2n)\equiv n+1\pmod{2n}\,,\quad n\in\mathbb{P}\setminus\{2\}\,.
\end{equation}
Recalling \eqref{zski}, $Z^{+}(2n)=\sum_{k=0}^{2n-1}\frac{(-1)^{k}+k!(2n-k-1)!}{2n}=\frac{1}{n}\sum_{k=0}^{n-1}k!(2n-k-1)!$ due to the even number of terms and consequential symmetry of the summand. Evidently, since $n\not\in\mathbb{P}_{1}\setminus\{2\}$ and if $n\neq 2$ and $n\neq 4$, $n^{2}\mid (2n-k-1)!$ for all $0\leq k\leq n-1$. Additionally, $2\mid k!$ for $2\leq k\leq n-1$ and $\frac{(2n-1)!+(2n-2)!}{2n^{2}}=\frac{(2n-2)!}{n}\in\mathbb{N}$ for $n\in\mathbb{N}$. With the cases $Z^{+}(4)=4$ and $Z^{+}(8)=1536$, this proves \eqref{c1}.

For the second congruence, again consider $Z^{+}(2n)=\frac{1}{n}\sum_{k=0}^{n-1}k!(2n-k-1)!$. Let us select distinct $0\leq k_{1}<k_{2}\leq n-1$ and take the sum of two terms $k_{1}!(2n-k_{1}-1)!+k_{2}!(2n-k_{2}-1)!=k_{1}!(2n-k_{2}-1)!\left(\frac{(2n-k_{1}-1)!}{(2n-k_{2}-1)!}+\frac{k_{2}!}{k_{1}!}\right)$. Setting $k_{2}=k_{1}+1$, this becomes $k_{1}!(2n-k_{1}-2)!(2n)$. Clearly, $2n^{2}\mid k_{1}!(2n-k_{1}-2)!(2n)$. Because $n$ is an odd prime, the sum for $Z^{+}(2n)$ will have an odd number of terms, allowing us to pair every other term with its succeeding term as described, leaving only the last one, $(n-1)!n!$, unpaired. Therefore, $nZ^{+}(2n)\equiv -n!\pmod{2n^{2}}$ by Wilson's theorem, and since Wilson quotients for odd primes are always odd, $2n^{2}\mid n(1+(n-1)!+n)$. A simple algebraic manipulation allows us to see that $-n!\equiv n^{2}+n\pmod{2n^{2}}$, which proves \eqref{c2}.

As for when $n$ is prime, we invoke \eqref{mpluscong} to observe that
\begin{multline}
Z^{+}(p)\equiv\sum_{k=0}^{p-1}(-1)^{k}\left(B_{2(p-1)}-B_{p-1}-H_{k}\right)\\
\equiv B_{2(p-1)}-B_{p-1}-\sum_{k=1}^{p-1}(-1)^{k}H_{k}\pmod{p}\,.
\end{multline}
Of course, $\sum_{k=1}^{p-1}(-1)^{k}H_{k}=-1+\left(1+\frac{1}{2}\right)-\left(1+\frac{1}{2}+\frac{1}{3}\right)+\ldots+\left(1+\ldots+\frac{1}{p-1}\right)=\frac{1}{2}\sum_{n=1}^{\frac{p-1}{2}}\frac{1}{n}=\frac{1}{2}H_{\frac{p-1}{2}}$, so
\begin{equation}
Z^{+}(p)\equiv B_{2(p-1)}-B_{p-1}-\frac{1}{2}H_{\frac{p-1}{2}}\pmod{p}
\end{equation}
when $p$ is an odd prime. If $p=2$, $Z^{+}(2)=1\equiv B_{2}-B_{1}+1\pmod{2}$. It is worth noting that this congruence bears a close resemblance to the one of $M_{\frac{p-1}{2}}(p)$. Thus, by combining \eqref{mcong} with the above congruence, we obtain the interesting relation
\begin{equation}
Z^{+}(p)\equiv\frac{1}{2}\left(M_{0}(p)+M_{\frac{p-1}{2}}(p)\right)\pmod{p}\,.
\end{equation}

\section{Generating Functions}

\subsection{Of $M$--numbers}

Considering the sequences of quotients $M_{k}(n)$ for $n\in\mathbb{N}$, its exponential generating function\footnote{Note that ordinary generating functions are of little utility in this instance, as their radius of convergence is zero.} for $\lvert x\rvert\leq R$, where $R>0$ is the radius of convergence, is derived as follows,
\begin{equation}\label{egf}
\operatorname{EG}(M_{k}(n);x)=\sum_{n=0}^{\infty}M_{k}(n+k+1)\frac{x^{n}}{n!}=\sum_{n=0}^{\infty}\left[\frac{1+(-1)^{k}k!n!}{n+k+1}\right]\frac{x^{n}}{n!}\,.
\end{equation}
Note that the first term appearing in the summand is $M_{k}(k+1)$, as $M_{k}(p)$ is undefined for $p\leq k$. Also observe that
\begin{equation}
\sum_{n=0}^{\infty}\left[\frac{1+(-1)^{k}k!n!}{n+k+1}\right]\frac{x^{n}}{n!}<\sum_{n=0}^{\infty}\left(1+(-1)^{k}k!\right)x^{n}\,,
\end{equation}
which has a radius of convergence $R=1$. Therefore, $\operatorname{EG}(M_{k}(n);x)$ has a positive radius of convergence of at least $1$. This allows the summation to be split as
\begin{equation}
\operatorname{EG}(M_{k}(n);x)=\sum_{n=0}^{\infty}\frac{x^{n}}{n!(n+k+1)}+(-1)^{k}k!\sum_{n=0}^{\infty}\frac{x^{n}}{n+k+1}\,,
\end{equation}
where the radius of convergence $R=1$ from the second summation. As dictated by the definitions of the lower incomplete gamma function $\gamma(a,x)=x^{a}\sum_{n=0}^{\infty}\frac{(-x)^{n}}{n!(n+a)}$ and the Lerch transcendent $\Phi(x,s,a)=\sum_{n=0}^{\infty}\frac{x^{n}}{(n+a)^{s}}$ \cite{olbc}, we derive the exponential generating function by setting $a=k+1$ and $s=1$,
\begin{equation}
\operatorname{EG}(M_{k}(n);x)=(-1)^{k}\left(k!\Phi(x,1,k+1)-\frac{\gamma(k+1,-x)}{x^{k+1}}\right)\,.
\end{equation}

For the exponential generating function of the signless $M^{+}_{k}(p)$, it is not difficult to see from \eqref{14} that
\begin{equation}
\operatorname{EG}(M^{+}_{k}(n);x)=(-1)^{k}\operatorname{EG}(M_{k}(n);x)=k!\Phi(x,1,k+1)-\frac{\gamma(k+1,-x)}{x^{k+1}}\,,
\end{equation}
since the summations are independent of $k$.

\subsection{Of $Z$--numbers}

We can also derive the exponential generating function of $Z(n)$ for $n\in\mathbb{N}$ and $\lvert x\rvert\leq R$, where $R>0$ is the radius of convergence,
\begin{equation}\label{egfz}
\operatorname{EG}(Z(n);x)=\sum_{n=0}^{\infty}Z(n+1)\frac{x^{n}}{n!}=\sum_{n=0}^{\infty}\left[1+\frac{\left(1+(-1)^{n}\right)n!}{n+2}\right]\frac{x^{n}}{n!}\,.
\end{equation}
Similarly, the sequence begins at $Z(1)$, as $Z(0)$ is undefined by its original definition. Since
\begin{equation}
\sum_{n=0}^{\infty}\left[1+\frac{\left(1+(-1)^{n}\right)n!}{n+2}\right]\frac{x^{n}}{n!}<\sum_{n=0}^{\infty}3x^{n}\,,
\end{equation}
which has a radius of convergence $R=1$, $\operatorname{EG}(Z(n);x)$ has a positive radius of convergence of at least $1$. Additionally, $\sum_{n=0}^{\infty}\frac{\left(1+(-1)^{n}\right)x^{n}}{n+2}=2\sum_{n=0}^{\infty}\frac{x^{2n}}{2n+2}=\sum_{n=0}^{\infty}\frac{x^{2n}}{n+1}$. This allows the summation to be split as
\begin{equation}
\operatorname{EG}(Z(n);x)=\sum_{n=0}^{\infty}\frac{x^{n}}{n!}+\sum_{n=0}^{\infty}\frac{x^{2n}}{n+1}\,,
\end{equation}
where the radius of convergence is clearly $R=1$. Of course, $\sum_{n=0}^{\infty}\frac{x^{n}}{n!}=e^{x}$, and because the natural logarithm has the Taylor expansion $\ln(z)=\sum_{n=1}^{\infty}(-1)^{n+1}\frac{(z-1)^{n}}{n}$ for $z\in (0,2]$, we see that
\begin{equation}
\ln(1-x^{2})=\sum_{n=1}^{\infty}(-1)^{n+1}\frac{(-x^{2})^{n}}{n}=-\sum_{n=1}^{\infty}\frac{x^{2n}}{n}=-x^{2}\sum_{n=1}^{\infty}\frac{x^{2n-2}}{n}=-x^{2}\sum_{n=0}^{\infty}\frac{x^{2n}}{n+1}\,.
\end{equation}
Therefore, the exponential generating function is obtained,
\begin{equation}
\operatorname{EG}(Z(n);x)=e^{x}-\frac{1}{x^{2}}\ln\left(1-x^{2}\right)\,.
\end{equation}

The exponential generating function for $Z^{+}(n)$ is a pinch more involved, though invoking the \eqref{zpluslerch} for $n\in\mathbb{N}$ and $\lvert x\rvert\leq R$, we obtain
\begin{multline}
\operatorname{EG}(Z^{+}(n);x)=\sum_{n=0}^{\infty}Z^{+}(n+1)\frac{x^{n}}{n!}\\
=\sum_{n=0}^{\infty}\left[\frac{1+(-1)^{n}}{2n+2}-n!\left(2\Phi(2,1,n+2)+2^{-n-1}i\pi\right)\right]\frac{x^{n}}{n!}\,.
\end{multline}
Since $\lvert\Phi(2,1,n+2)\rvert<1$ for $n\geq 1$, we can say $\operatorname{EG}(Z^{+}(n);x)<\sum_{n=0}^{\infty}7x^{n}$, which has a radius of convergence of $1$, and thus our generating function converges for $R\geq 1$. Therefore, under splitting, simplification, and removal of null terms,
\begin{equation}
\operatorname{EG}(Z^{+}(n);x)=2\sum_{n=0}^{\infty}\frac{x^{2n}}{(2n)!(2n+2)}-2\sum_{n=0}^{\infty}\Phi(2,1,n+2)x^{n}-\frac{i\pi}{2}\sum_{n=0}^{\infty}\left(\frac{x}{2}\right)^{n}\,.
\end{equation}
For the first summation, let us take the series expansions of hyperbolic trigonometric functions $\sinh x=\sum_{n=0}^{\infty}\frac{x^{2n+1}}{(2n+1)!}$ and $\cosh x=\sum_{n=0}^{\infty}\frac{x^{2n}}{(2n)!}$ to derive that
\begin{equation}
\cosh x-1=-1+\sum_{n=0}^{\infty}\frac{x^{2n}}{(2n)!}=\sum_{n=1}^{\infty}\frac{x^{2n}}{(2n)!}=\sum_{n=0}^{\infty}\frac{x^{2n+2}}{(2n+2)!}\,,
\end{equation}
and thus
\begin{equation}
x\sinh x -\cosh x+1=\sum_{n=0}^{\infty}\left[\frac{x^{2n+2}}{(2n+1)!}-\frac{x^{2n+2}}{(2n+2)!}\right]=x^{2}\sum_{n=0}^{\infty}\frac{x^{2n}}{(2n)!(2n+2)}\,,
\end{equation}
which is precisely our series. Additionally, it is useful to recall that $\sum_{n=0}^{\infty}x^{n}$ generates $\frac{1}{1-x}$. Putting these together, we obtain the exponential generating function
\begin{equation}
\operatorname{EG}(Z^{+}(n);x)=\frac{2}{x^{2}}(x\sinh x-\cosh x+1)+\frac{i\pi}{x-2}-2\operatorname{G}(\Phi(2,1,n+2);x)\,,
\end{equation}
where $\operatorname{G}(\Phi(2,1,n+2);x)$ is the generating function of $\Phi(2,1,n+2)$. For the sake of notation, let us define $\tau(z,s,\alpha):=\operatorname{G}(\Phi(z,s,\alpha);x)$, so
\begin{equation}
\operatorname{EG}(Z^{+}(n);x)=\frac{2}{x^{2}}(x\sinh x-\cosh x+1)+\frac{i\pi}{x-2}-2\tau(2,1,n+2)\,.
\end{equation}

\section{Tables}

\begin{center}
$$Z(n)$$
\begin{tabular}{||c | c c c c c c c c c c c c c||} 
 \hline
 $n:$ & $1$ & $2$ & $3$ & $4$ & $5$ & $6$ & $7$ & $8$ & $9$ & $10$ & $11$ & $12$ & $13$ \\
 \hline\hline
 $*$ & $2$ & $1$ & $2$ & $1$ & $9$ & $1$ & $181$ & $1$ & $8065$ & $1$ & $604801$ & $1$ & $68428801$ \\
 \hline
\end{tabular}
\begin{tabular}{||c | c c c c c c c||} 
 \hline
 $n:$ & $14$ & $15$ & $16$ & $17$ & $18$ & $19$ & $20$ \\
 \hline\hline
 $*$ & $1$ & $10897286401$ & $1$ & $2324754432001$ & $1$ & $640237370572801$ & $1$ \\
 \hline
\end{tabular}
\end{center}

\begin{center}
$$Z^{+}(n)$$
\begin{tabular}{||c | c c c c c c c c c c c c||} 
 \hline
 $n:$ & $1$ & $2$ & $3$ & $4$ & $5$ & $6$ & $7$ & $8$ & $9$ & $10$ & $11$ & $12$ \\
 \hline\hline
 $*$ & $2$ & $1$ & $2$ & $4$ & $13$ & $52$ & $259$ & $1536$ & $\frac{95617}{9}$ & $84096$ & $750371$ & $7453440$ \\
 \hline
\end{tabular}
\begin{tabular}{||c | c c c c c||} 
 \hline
 $n:$ & $13$ & $14$ & $15$ & $16$ & $17$ \\
 \hline\hline
 $*$ & $81566917$ & $974972160$ & $\frac{189550368001}{15}$ & $176504832000$ & $2642791002353$ \\
 \hline
\end{tabular}
\end{center}
\newpage
\begin{center}
$$M_{k}(n)$$
\begin{tabular}{||c | c c c c c c c c c c||} 
 \hline
 $k$ & $n:1$ & $2$ & $3$ & $4$ & $5$ & $6$ & $7$ & $8$ & $9$ & $10$ \\
 \hline\hline
 $0$ & $2$ & $1$ & $1$ & $\frac{7}{4}$ & $5$ & $\frac{121}{6}$ & $103$ & $\frac{5041}{8}$ & $\frac{40321}{9}$ & $\frac{362881}{10}$ \\
 \hline
 $1$ & $*$ & $0$ & $0$ & $-\frac{1}{4}$ & $-1$ & $-\frac{23}{6}$ & $-17$ & $-\frac{719}{8}$ & $-\frac{5039}{9}$ & $-\frac{40319}{10}$ \\
 \hline
 $2$ & $*$ & $*$ & $1$ & $\frac{3}{4}$ & $1$ & $\frac{13}{6}$ & $7$ & $\frac{241}{8}$ & $\frac{1441}{9}$ & $\frac{10081}{10}$ \\
 \hline
 $3$ & $*$ & $*$ & $*$ & $-\frac{5}{4}$ & $-1$ & $-\frac{11}{6}$ & $-5$ & $-\frac{143}{8}$ & $-\frac{719}{9}$ & $-\frac{4319}{10}$ \\
 \hline
 $4$ & $*$ & $*$ & $*$ & $*$ & $5$ & $\frac{25}{6}$ & $7$ & $\frac{145}{8}$ & $\frac{577}{9}$ & $\frac{2881}{10}$ \\
 \hline
 $5$ & $*$ & $*$ & $*$ & $*$ & $*$ & $-\frac{119}{6}$ & $-17$ & $-\frac{239}{8}$ & $-\frac{719}{9}$ & $-\frac{2879}{10}$ \\
 \hline
 $6$ & $*$ & $*$ & $*$ & $*$ & $*$ & $*$ & $103$ & $\frac{721}{8}$ & $\frac{1441}{9}$ & $\frac{4321}{10}$ \\
 \hline
 $7$ & $*$ & $*$ & $*$ & $*$ & $*$ & $*$ & $*$ & $-\frac{5039}{8}$ & $-\frac{5039}{9}$ & $-\frac{10079}{10}$ \\
 \hline
 $8$ & $*$ & $*$ & $*$ & $*$ & $*$ & $*$ & $*$ & $*$ & $\frac{40321}{9}$ & $\frac{40321}{10}$ \\
 \hline
 $9$ & $*$ & $*$ & $*$ & $*$ & $*$ & $*$ & $*$ & $*$ & $*$ & $-\frac{362879}{10}$ \\
 \hline
\end{tabular}

\begin{tabular}{||c | c c c c c||} 
 \hline
 $k$ & $n:11$ & $12$ & $13$ & $14$ & $15$ \\
 \hline\hline
 $0$ & $329891$ & $\frac{39916801}{12}$ & $36846277$ & $\frac{6227020801}{14}$ & $\frac{87178291201}{15}$ \\
 \hline
 $1$ & $-32989$ & $-\frac{3628799}{12}$ & $-3070523$ & $-\frac{479001599}{14}$ & $-\frac{6227020799}{15}$ \\
 \hline
 $2$ & $7331$ & $\frac{725761}{12}$ & $558277$ & $\frac{79833601}{14}$ & $\frac{958003201}{15}$ \\
 \hline
 $3$ & $-2749$ & $-\frac{241919}{12}$ & $-167483$ & $-\frac{21772799}{14}$ & $-\frac{239500799}{15}$ \\
 \hline
 $4$ & $1571$ & $\frac{120961}{12}$ & $74437$ & $\frac{8709121}{14}$ & $\frac{87091201}{15}$ \\
 \hline
 $5$ & $-1309$ & $-\frac{86399}{12}$ & $-46523$ & $-\frac{4838399}{14}$ & $-\frac{43545599}{15}$ \\
 \hline
 $6$ & $1571$ & $\frac{86401}{12}$ & $39877$ & $\frac{3628801}{14}$ & $\frac{29030401}{15}$ \\
 \hline
 $7$ & $-2749$ & $-\frac{120959}{12}$ & $-46523$ & $-\frac{3628799}{14}$ & $-\frac{25401599}{15}$ \\
 \hline
 $8$ & $7331$ & $\frac{241921}{12}$ & $74437$ & $\frac{4838401}{14}$ & $\frac{29030401}{15}$ \\
 \hline
 $9$ & $-32989$ & $-\frac{725759}{12}$ & $-167483$ & $-\frac{8709119}{14}$ & $-\frac{43545599}{15}$ \\
 \hline
 $10$ & $329891$ & $\frac{3628801}{12}$ & $558277$ & $\frac{21772801}{14}$ & $\frac{87091201}{15}$ \\
 \hline
 $11$ & $*$ & $-\frac{39916799}{12}$ & $-3070523$ & $-\frac{79833599}{14}$ & $-\frac{239500799}{15}$ \\
 \hline
 $12$ & $*$ & $*$ & $36846277$ & $\frac{479001601}{14}$ & $\frac{958003201}{15}$ \\
 \hline
 $13$ & $*$ & $*$ & $*$ & $-\frac{6227020799}{14}$ & $-\frac{6227020799}{15}$ \\
 \hline
 $14$ & $*$ & $*$ & $*$ & $*$ & $\frac{87178291201}{15}$ \\
 \hline
\end{tabular}
\end{center}

\end{document}